\documentclass[12pt]{elsarticle}
\usepackage{amsmath, amsthm, amssymb, amsfonts, color, enumerate}

\newtheorem{theorem}{Theorem}[section]

\newtheorem{claim}[theorem]{Claim}

\newtheorem{conjecture}[theorem]{Conjecture}

\newtheorem{lemma}[theorem]{Lemma}

\newtheorem{problem}[theorem]{Problem}
\newtheorem{proposition}[theorem]{Proposition}

\usepackage{fullpage}

\journal{Journal of Combinatorial Theory, Series A}

\begin{document}

\begin{frontmatter}

\title{The Colin de Verdi\`ere parameter, excluded minors, and the spectral radius}
\author{Michael Tait\corref{cor1}}\ead{mtait@cmu.edu}
\cortext[cor1]{Research supported by NSF grant DMS-1606350.}

\address{Department of Mathematical Sciences\\
Carnegie Mellon University\\
5000 Forbes Avenue, Pittsburgh, PA 15213, USA}
\date{}

\begin{abstract}
In this paper we characterize graphs which maximize the spectral radius of their adjacency matrix over all graphs of Colin de Verdi\`ere parameter at most $m$. We also characterize graphs of maximum spectral radius with no $H$ as a minor when $H$ is either $K_r$ or $K_{s,t}$. Interestingly, the extremal graphs match those which maximize the number of edges over all graphs with no $H$ as a minor when $r$ and $s$ are small, but not when they are larger.
\end{abstract}
\end{frontmatter}

\section{Introduction}
Let $H$ and $G$ be simple graphs. $H$ is a {\em minor} of $G$ if $H$ can be obtained from a subgraph of $G$ by contracting edges. Properties of graphs with excluded minors have been studied extensively. In particular, Mader \cite{Mader} proved that for every graph $H$ there is a constant $C$ such that if $G$ does not contain $H$ as a minor, then $|E(G)| \leq C|V(G)|$. Determining this constant $C$ seems to be a very difficult question for general $H$.

When $H$ is either a complete graph or a complete bipartite graph, there are natural constructions of graphs $G$ which do not contain $H$ as a minor. A complete graph on $r-2$ vertices joined to an independent set of size $n-r+2$ is an $n$ vertex graph with $(r-2)(n-r+2) + \frac{1}{2}(r-2)(r-3)$ edges which does not contain $K_r$ as a minor. A complete graph on $s-1$ vertices joined to $(n-s+1)/t$ disjoint copies of $K_t$ is an $n$-vertex graph with $\frac{1}{2}(t+2s-3)(n-s+1) + \frac{1}{2}(s-1)(s-2)$ edges which does not contain $K_{s,t}$ as a minor. 

Mader \cite{Mader2} showed that this construction yields the maximum number of edges over all $n$-vertex graphs with no $K_r$ minor when $r\leq 7$. Surprisingly, this natural construction is not best possible when $r > 7$. Indeed, Kostochka and Thomason \cite{Kostochka1, Kostochka2, Thomason1, Thomason2} independently showed that the maximum number of edges in a $K_r$ minor-free graph is $\Theta(r(\log r)^{1/2} n)$ for large $r$. Similarly, Chudnovsky, Reed, and Seymour \cite{ChudnovskyReedSeymour} showed that the graph on $n$ vertices with no $K_{2,t}$ minor with the maximum number of edges is given by this natural construction (see also \cite{Myers}). Kostochka and Prince \cite{KostochkaPrince} showed that the same is true when $s=3$ and $t$ is large enough. It is unknown when $s \in \{4,5\}$ and Kostochka and Prince \cite{KostochkaPrince} showed that this construction is not best possible when $s\geq 6$.

In this paper we discuss a related question. If $G$ is an $n$-vertex graph with no $H$ minor, and $\lambda$ is the spectral radius of its adjacency matrix, how large can $\lambda$ be? We show that the natural $K_r$ and $K_{s,t}$ minor-free graphs described above are extremal for all values of $r$ and $s\leq t$. It is interesting that the extremal graphs for maximizing number of edges and spectral radius are the same for small values of $r$ and $s$ and then differ significantly. We also consider maximizing the spectral radius over the family of $n$-vertex graphs of Colin de Verdi\`ere parameter at most $m$. Our main theorems are the following.

\begin{theorem}\label{CdV}
Let $m\geq 2$ be an integer. For $n$ large enough, the $n$-vertex graph of maximum spectral radius with Colin de Verdi\`ere parameter at most $m$ is the join of $K_{m-1}$ and a path of length $n-m+1$.
\end{theorem}

\begin{theorem}\label{K_r}
Let $r\geq 3$. For $n$ large enough, the $n$-vertex graph with no $K_r$ minor of maximum spectral radius is the join of $K_{r-2}$ and an independent set of size $n-r+2$.
\end{theorem}

\begin{theorem}\label{K_s,t}
Let $t\geq s\geq 2$. For $n$ large enough, if $G$ is an $n$-vertex graph with no $K_{s,t}$ minor and $\lambda$ is the spectral radius of its adjacency matrix, then 
\[
\lambda \leq \frac{s+t-3 + \sqrt{(s+t-3)^2 + 4((s-1)(n-s+1) - (s-2)(t-1))}}{2}
\]
with equality if and only if $n \equiv s-1 \pmod{t}$ and $G$ is the join of $K_{s-1}$ with a disjoint union of copies of $K_t$.
\end{theorem}

The Colin de Verdi\`ere parameter of a graph, denoted by $\mu(G)$, was introduced in 1990 \cite{Colin1} motivated by applications in differential geometry. This parameter turned out to have many nice graph-theoretic properties, for instance it is minor-monotone. Further, much work has been done relating $\mu$ to other graph parameters  (for a nice survey, see \cite{barioli survey}), and it can be used to give an algebraic description of certain topological properties of a graph (proved by Colin de Verdi\`ere \cite{Colin1}, Robertson, Seymour, and Thomas \cite{RST}, and Lov\'asz and Schrijver \cite{LS}):

\begin{enumerate}[(i)]
\item $\mu(G) \leq 1$ if and only if $G$ is the disjoint union of paths.
\item $\mu(G) \leq 2$ if and only if $G$ is outerplanar.
\item $\mu(G) \leq 3$ if and only if $G$ is planar.
\item $\mu(G) \leq 4$ if and only if $G$ is linklessly embeddable.
\end{enumerate}

Our theorems build on work of Tait and Tobin \cite{TaitTobin}, of Nikiforov \cite{Nikiforov}, and of Hong \cite{Hong}. In \cite{TaitTobin}, planar and outerplanar graphs of maximum spectral radius were characterized for $n$ large enough. By the above characterization, Theorem \ref{CdV} is a far-reaching generalization of the main results of \cite{TaitTobin}. In \cite{Nikiforov} graphs with no $K_{2,t}$ minor of maximum spectral radius were characterized and in \cite{Hong} graphs with no $K_5$ minor of maximum spectral radius were characterized. Theorems \ref{K_r} and \ref{K_s,t} extend these results to all $r\geq 3$ and $2\leq s\leq t$. In \cite{Hong} the maximum spectral radius of all graphs with a fixed tree-width was determined, and Theorem \ref{K_r} shows that the extremal graph is the same as for the family of graphs with no $K_r$ minor. We also mention the following open problem, which would be implied by a solution to a conjecture of Nevo \cite{Nevo}:

\begin{problem}
Show that if $\mu(G) \leq m$, then $e(G) \leq mn - \binom{m+1}{2}$.
\end{problem}

Finally, we note that finding the graph in a given family of graphs maximizing some function of its eigenvalues has a long history in extremal graph theory (for example Stanley's bound \cite{Stanley}, the Alon-Boppana theorem \cite{AlonBoppana}, the Hoffman ratio bound \cite{Hoffman}). In particular, in some cases theorems of this type can strengthen classical extremal results, for example Tur\'an's theorem \cite{NikiforovTuran} or the K\H{o}vari-S\'os-Tur\'an theorem \cite{BabaiGuiduli, NikiforovKST}.

\subsection{Notation, definitions, and outline}
If $G$ is a graph, $e(G)$ will denote the number of edges in $G$. $A(G)$ will denote the adjacency matrix of $G$ and $\lambda_1(A(G))$ or $\lambda_1(G)$ will denote the largest eigenvalue of this matrix. For two sets $S,L \subset V(G)$, $E(S)$ will denote the edges with both endpoints in $S$ and $E(S,L)$ will denote edges with one endpoint in $S$ and one in $L$. We will often use that $n$ is large enough and that if $G$ is a connected graph, then the eigenvector corresponding to $\lambda_1(G)$ has all positive entries. This fact and the Rayleigh quotient imply that if $H$ is a strict subgraph of $G$, then $\lambda_1(H) < \lambda_1(G)$.

Given a matrix $M$, define the {\em corank} of $M$ to be the dimension of its kernel. If $G$ is an $n$-vertex graph, then the {\em Colin de Verdi\`ere parameter} of $G$, denoted by $\mu(G)$, is defined to be the largest corank of any $n\times n$ matrix $M$ such that:

\begin{enumerate}[M1]
\item If $i\not=j$ then $M_{ij} < 0$ if $i\sim j$ and $M_{ij} = 0$ if $i\not\sim j$.
\item $M$ has exactly one negative eigenvalue of multiplicity $1$.
\item There is no nonzero matrix $X$ such that $MX = 0$ and $X_{ij}  = 0$ whenever $i=j$ or $M_{ij} \not=0$.
\end{enumerate}

A nice discussion of where these seemingly {\em ad hoc} conditions come from is given in \cite{survey}. Two key properties that we will use (see \cite{survey}) are that if $\mu(G) \leq m$ then there is a finite family of minors that $G$ does not contain and there is a constant $c_m$ such that $e(G) \leq c_m n $.

In section \ref{structural section} we prove structural results about graphs with excluded minors which we will need during the proof of the main theorems. These results are relatively specific, but may be of independent interest. In section \ref{section CdV} we prove Theorem \ref{CdV}. In section \ref{section K_r} we prove Theorem \ref{K_r} and in section \ref{section K_s,t} we prove Theorem \ref{K_s,t}.

\section{Structural results for graphs with excluded minors}\label{structural section}
We need several structural results for graphs which do not contain a fixed minor:

\begin{lemma}\label{K_s,t bipartite bound}
Let $G$ be a bipartite $n$ vertex graph with no $K_{s,t}$ minor and vertex partition $K$ and $T$. Let $|K| = k$ and $|T| = n-k$. Then there is a constant $C$ depending only on $s$ and $t$ such that 
\[
e(G) \leq Ck + (s-1)n.
\]
In particular, if $|K| = o(n)$, then $e(G) \leq (s-1 + o(1))n$.
\end{lemma}

This lemma has been proved several times in the literature, for example \cite{Thomason3} Theorem 2.2. The next lemma follows from it, since a $K_{r-1,r-1}$ minor contains a $K_r$ minor.





\begin{lemma}\label{K_r bipartite bound}
Let $G$ be a bipartite $K_r$ minor-free graph on $n$ vertices with vertex partition $K$ and $T$. Let $|K| = k$ and $|T| = n-k$. Then there is an absolute constant $C$ depending only on $r$ such that
\[
e(G) \leq Ck + (r-2)n.
\]
In particular, if $|K| = o(n)$, then $e(G) \leq (r-2+o(1))n$.
\end{lemma}



If a graph is linklessly embeddable, it does not contain $K_{4,4}$ as a minor. Hence, Lemma \ref{K_s,t bipartite bound} implies that a bipartite linklessly embeddable graph with $o(n)$ vertices in one partite set has at most $(3+o(1))n$ edges. The following problem is open, and would be implied by a solution to Conjecture 4.5 in \cite{bipartite rigidity}:

\begin{problem}
If $G$ is a bipartite linkessly embeddable graph, show that $e(G) \leq 3n-9$.
\end{problem}

\begin{theorem}\label{K_r clique}
Let $G$ be a graph on $n$ vertices with no $K_r$ minor. Assume that $(1-2\delta)n > r$, and $(1-\delta)n > \binom{r-2}{2}+2$, and that there is a set $K$ with $|K| = r-2$ and a set $T$ with $|T| = (1-\delta)n$ such that every vertex in $K$ is adjacent to every vertex in $T$. Then we may add edges to $K$ to make it a clique and the resulting graph will still have no $K_r$ minor.
\end{theorem}

\begin{proof}
\begin{claim}
$T$ induces an independent set.
\end{claim}
To see this, suppose that there are $u,v\in T$ that are adjacent. For each pair $x,y \in K$, choose $b_{xy} \in T$ all distinct and different from $u$ and $v$ (this can be done as $|T| > \binom{r-2}{2} + 2$ by assumption). Then the paths $x-b_{xy}-y$ along with $u$ and $v$ form a subdivision of $K_r$, a contradiction.
\begin{claim}
Let $C$ be a component of $G\setminus (K\cup T)$. Then there is at most one vertex in $T$ with a neighbor in $C$.
\end{claim}
Suppose $u,v \in T$ each have a neighbor in $C$. Then there is a path from $u$ to $v$ where all the interior vertices are in $C$. Choose $b_{xy}$ as before. Then this path along with the paths $x-b_{xy}-y$ form a subdivision of $K_r$. This is a contradiction, proving the claim.

\medskip

Now let $D$ be the set of vertices in $T$ that have degree exactly $r-2$ in $G$. Since there are at most $\delta n$ components in $G\setminus (K\cup T)$, we have that $|D| \geq (1-2\delta)n$. Now add edges to $K$ so that $K$ induces a clique. Assume that there is now a $K_r$ minor in $G$. Consequently, there are $r$ disjoint sets $X_1,\cdots, X_r$ of vertices such that 
\begin{enumerate}[(a)]
\item\label{a} For each $X_i$, either $G[X_i]$ is connected or every component of $G[X_i]$ intersects $K$.
\item\label{b} For all distinct $i,j$ either there is an edge of $G$ between $X_i$ and $X_j$, or both $X_i$ and $X_j$ have nonempty intersection with $K$.
\end{enumerate}

Choose $X_1,\cdots, X_r$ such that as many of them as possible have nonempty intersection with $D$. Since the vertices of $D$ all have the same neighborhood, we can choose $X_1,\cdots, X_r$ such that each $X_i$ contains at most one vertex from $D$. Since we assumed $(1-2\delta)n > r$ there is a vertex $v\in D$ which is not in any of the sets $X_1,\cdots, X_r$.

\begin{claim}\label{3}
For $1\leq i\leq r$, if $X_i$ is disjoint from $D$ then $X_i$ is disjoint from $K$.
\end{claim}
If there exists $X_i$ that is disjoint from $D$ but intersects $K$, then we can add $v$ to $X_i$ giving a better choice of $X_1, \cdots, X_r$.

\begin{claim}
$G[X_i]$ is connected for all $i$.
\end{claim}
Assume for the sake of contradiction that $G[X_i]$ is disconnected. Then by \eqref{a}, we have that each component of $G[X_i]$ intersects $K$. Claim \ref{3} implies that $X_i$ intersects $D$. But if each component intersects $K$ and $X_i$ contains a vertex in $D$, then $G[X_i]$ is connected.

\begin{claim}\label{5}
For any distinct $i$ and $j$, there is an edge of $G$ between $X_i$ and $X_j$.
\end{claim}
Assume for the sake of contradiction that there is an $i$ and $j$ with no edge between $X_i$ and $X_j$. Then \eqref{b} implies that both $X_i$ and $X_j$ have nonempty intersection with $K$ and Claim \ref{3} implies that they both intersect $D$. Therefore there is an edge (between $K$ and $D$, ie in $E(G)$) between them.

\medskip

Now $X_1, \cdots, X_r$ form a $K_r$ minor in $G$, a contradiction.
\end{proof}

\begin{theorem}\label{K_s,t clique}
Let $\mathcal{H}$ be a fixed family of graphs and let $G$ be an $n$-vertex graph with no minor from $\mathcal{H}$. Assume further that there is a set $K$ of size $s-1$ and a set $T$ with $|T| = (1-\delta)n > (s-1)(s-2)/2$ and that every vertex in $K$ is adjacent to every vertex in $T$. Let $c$ be such that every graph of average degree $c$ has some $H\in \mathcal{H}$ as a minor. Then there is a set of at most 
\[
\frac{c(s-1)(s-2)}{2(1-\delta)}
\]
edges such that after deleting these edges we can make $K$ a clique without introducing a minor from $\mathcal{H}$.
\end{theorem}

\begin{proof}
Let $d$ be the average degree of vertices in $T$. $G$ has at most $cn/2$ edges, so $d|T|/2 \leq cn/2$ which implies that 
\[
d\leq \frac{c}{1-\delta}.
\]
Let $m= (s-1)(s-2)/2$ and let $M$ be a set of $m$ vertices in $T$ with sum of degrees as small as possible. By the first moment method, there are at most 
\[
\frac{c(s-1)(s-2)}{2(1-\delta)}
\]
edges in $G$ that are incident with $M$. For each pair $xy$ of vertices in $K$, choose a vertex $b_{xy}\in M$ such that all $b_{xy}$ are distinct. For each $x,y$, delete all edges adjacent to $b_{xy}$ except for the edges to $x$ and to $y$. Let $G'$ be the subgraph of $G$ produced, and note that $G'$ does not have a minor from $\mathcal{H}$. Now, $x$ and $y$ are nonadjacent vertices in $K$, then they are the neighbors of a vertex of degree $2$ in $G'$. Thus making $x$ adjacent to $y$ cannot introduce a minor in $G'$ unless it was already present.
\end{proof}

\section{Graphs with no $K_r$ minor}\label{section K_r}
Let $G_r$ be an $n$ vertex graph with no $K_r$ minor which has maximum spectral radius of its adjacency matrix among all such graphs. In this section, we will prove Theorem \ref{K_r}, that $G_r$ is the join of $K_{r-2}$ and an independent set of size $n-r+2$ for sufficiently large $n$. Let $A$ be the adjacency matrix of $G_r$ and let $\lambda$ be the largest eigenvalue of $A$. Let $\mathbf{x}$ be an eigenvector for $\lambda$. Without loss of generality, we may assume the $G_r$ is connected and so $\mathbf{x}$ is well-defined. We will assume throughout this section that $\mathbf{x}$ is normalized to have maximum entry equal to $1$, and that $z$ is a vertex such that $\mathbf{x}_z = 1$ (if there is more than one such vertex, choose $z$ arbitrarily). We will use throughout the section that $e(G_r) = O(n)$ since $G_r$ has no $K_r$ minor. The outline of our proof is as follows:
\begin{enumerate}
\item First we show that if a vertex has eigenvector entry close to $1$, then it has degree close to $n$ (Lemma \ref{lemma non-neighbors in S}).
\item We show that there are $r-2$ vertices of eigenvector entry close to $1$, and hence degree close to $n$.
\item We use Theorem \ref{K_r clique} to show that these $r-2$ vertices induce a clique.
\item We show that each of the $r-2$ vertices in the clique actually have degree $n-1$.
\end{enumerate}

First, we split the vertex set into vertices with ``large" eigenvector entry and those with ``small". Let
\[
L = \{v\in V(G) : \mathbf{x}_v > \epsilon\},
\]
and 
\[
S = \{v \in V(G): \mathbf{x}_v \leq \epsilon\}
\]
where $\epsilon$ will be chosen later.

\begin{lemma}\label{lower bound on lambda}
$\sqrt{(r-2)(n-r+2)} \leq \lambda =  O\left(\sqrt{n}\right)$.
\end{lemma}
\begin{proof}
$K_{r-2, n-r+2}$ has no $K_r$ minor. Since $G_r$ is extremal, $\lambda > \lambda_1(A(K_{r-2, n-r+2})) = \sqrt{(r-2)(n-r+2)}$. For the upper bound, since $e(G_r) = O(n)$, the equality $2e(G) = \sum_{i=1}^n \lambda_i^2$ implies that $\lambda \leq \sqrt{2 e(G_r)} = O(\sqrt{n})$.
\end{proof}

The next lemma shows that $L$ is not too large.

\begin{lemma}\label{size of L}
\[
|L| = O\left( \sqrt{n} \right).
\]
\end{lemma}
\begin{proof}
We sum the eigenvector eigenvalue equation over all vertices in $L$:
\[
\lambda\sum_{u\in L} \mathbf{x}_u = \sum_{u\in L} \sum_{v\sim u} \mathbf{x}_v \leq 2e(G_r),
\]
where the last inequality follows because we have normalized so that each eigenvector entry is at most $1$. Now using that $e(G_r) = O(n)$ and $\lambda = \Omega(\sqrt{n})$ gives the result.

\end{proof}

Lemma \ref{size of L} and Lemma \ref{K_r bipartite bound} imply that for $n$ large enough, we have 
\begin{equation}\label{e(S,L) bound}
e(S,L) \leq (r-2+\epsilon)n.
\end{equation}

Next we use \eqref{e(S,L) bound} and the eigenvector-eigenvalue equation to give a bound on the sum over all eigenvector entries from $S$ and $L$.

\begin{lemma}\label{S and L entries bound}
There is an absolute constant $C_1$ depending on $R$ such that 
\[
\sum_{u\in S} \mathbf{x}_u \leq (1+C_1\epsilon) \sqrt{(r-2)(n-r+2)}
\]
and
\[
\sum_{u\in L} \mathbf{x}_u \leq C_1\epsilon\sqrt{(r-2)(n-r+2)}.
\]
\end{lemma}

\begin{proof}
For the first inequality, using the eigenvector-eigenvalue equation and summing over vertices in $S$ gives
\begin{align*}
\lambda \sum_{u\in S} \mathbf{x}_u &= \sum_{u\in S} \sum_{v\sim u} \mathbf{x}_v = \sum_{u\in S} \sum_{\substack{v\sim u \\ v\in S}} \mathbf{x}_v + \sum_{u\in S}\sum_{\substack{v\sim u \\ v\in L}} \mathbf{x}_v \\
& \leq  \sum_{u\in S} \sum_{\substack{v\sim u \\ v\in S}} \epsilon + \sum_{u\in S}\sum_{\substack{v\sim u \\ v\in L}} 1 \leq \epsilon \cdot 2e(S) + e(S,L).
\end{align*}
Using Lemma \ref{S and L entries bound}, that $e(S) = O(n)$, and that Lemma \ref{lower bound on lambda} proves the inequality.

The second inequality is similar:
\begin{equation*}
\lambda \sum_{u\in L} \mathbf{x}_u = \sum_{u\in L} \sum_{v\sim u} \mathbf{x}_v = \sum_{u\in L}\sum_{\substack{v\sim u \\ v\in S}} \mathbf{x}_v + \sum_{u\in L}\sum_{\substack{v\sim u \\ v\in L}} \mathbf{x}_v \leq \epsilon e(S,L) + 2e(L).
\end{equation*}
Lemma \ref{size of L} implies that $e(L) = O(\sqrt{n})$ and noting that $\lambda = \Omega(\sqrt{n})$ completes the proof.
\end{proof}

Now we would like to show that if a vertex has eigenvector entry close to $1$, then it must be adjacent to most of the vertices in $S$. Let $u\in L$. Then
\begin{align*}
&\sqrt{(r-2)(n-r+2)} \mathbf{x}_u \leq \lambda \mathbf{x}_u  = \sum_{v\sim u} \mathbf{x}_v \leq \sum_{v\in L} \mathbf{x}_v + \sum_{\substack{v\sim u \\ v\in S}} \mathbf{x}_v \\
&= \sum_{v\in L} \mathbf{x}_v + \left(\sum_{v\in S} \mathbf{x}_v - \sum_{\substack{v\not\sim u \\ v\in S}} \mathbf{x}_v\right) \\
& \leq C_1 \epsilon \sqrt{(r-2)(n-r+2)} + \left((1+C_1\epsilon)\sqrt{(r-2)(n-r+2)} -  \sum_{\substack{v\not\sim u \\ v\in S}} \mathbf{x}_v\right).
\end{align*}

That is,
\begin{equation}\label{eq non-neighbors in S}
\sum_{\substack{v\not\sim u\\v\in S}} \mathbf{x}_v \leq (1 + 2C_1\epsilon - \mathbf{x}_u)\sqrt{(r-2)(n-r+2)}.
\end{equation}

This equation says that if $\mathbf{x}_u$ is close to $1$, then the sum of eigenvector entries over all vertices in $S$ not adjacent to $u$ is not too big. In order to show that this implies $u$ is adjacent to most of the vertices in $S$, we need an easy lower bound on the eigenvector entries of the vertices in $V(G_r)$.

\begin{claim}\label{lower bound eigenvector entry}
There is an absolute constant $C_2$ such that for all $u\in V(G_r)$, $\mathbf{x}_u\geq  \frac{1}{C_2\sqrt{(r-2)(n-r+2)}}$.
\end{claim}

\begin{proof}
Let $u\in V(G_r)$ be any vertex that is not $z$ (we already know that $z$ satisfies this inequality). If $u \sim z$, then $\lambda \mathbf{x}_u = \sum_{v\sim u} \mathbf{x}_v \geq \mathbf{x}_z$ and the inequality is satisfied. 

If not, assume that $\mathbf{x}_u \leq \frac{1}{C_2\sqrt{(r-2)(n-r+2)}}$. Then 
\[
\sum_{v\sim u} \mathbf{x}_v = \lambda \mathbf{x}_u \leq \frac{O(1)}{C_2}
\] 
 where the last inequality holds by the upper bound in Lemma \ref{lower bound on lambda}. Let $H$ be the graph obtained by removing all edges incident with $u$ and creating one new edge $uz$. Let the adjacency matrix of $H$ be $B$ and let $\mu$ be the spectral radius of $B$. Note that adding a leaf to a graph with no $K_r$ minor cannot produce a $K_r$ minor, and so $H$ is $K_r$ minor-free. Now since $\mu = \max_{\mathbf{z}\not= 0} \frac{\mathbf{z}^TB\mathbf{z}}{\mathbf{z}^T\mathbf{z}}$, we have 
 \begin{align*}
 \mu - \lambda \geq \frac{\mathbf{x}^T B \mathbf{x}}{\mathbf{x}^T\mathbf{x}} - \frac{\mathbf{x}^T A \mathbf{x}}{\mathbf{x}^T\mathbf{x}} = \frac{2}{\mathbf{x}^T \mathbf{x}}\left( \mathbf{x}_u\mathbf{x}_z - \mathbf{x}_u\sum_{v\sim u} \mathbf{x}_v\right) = \frac{2 \mathbf{x}_u}{\mathbf{x}^T\mathbf{x}}\left(1 - \frac{O(1)}{C_2}\right).
 \end{align*}
 Since $G_r$ is extremal, $\mu - \lambda \leq 0$, a contradiction for a large enough constant $C_2$.

\end{proof}

Now we can give a bound on the number of vertices in $S$ not adjacent to $u$.

\begin{lemma}\label{lemma non-neighbors in S}
Let $A_u = \{v\in S: v\not\sim u\}$ and assume that $\mathbf{x}_u = 1-\delta$. Then there is an absolute constant $C_3$ such that
\[
|A_u| \leq C_3 (\delta + \epsilon) n
\]
\end{lemma}
\begin{proof}
Applying \eqref{eq non-neighbors in S} and Claim \ref{lower bound eigenvector entry} yields
\[
|A_u| \leq C_2(1+2C_1 \epsilon - \mathbf{x}_u)(r-2)(n-r+2).
\]
\end{proof}

Lemma \ref{lemma non-neighbors in S} shows that the number of neighbors of $z$ tends to $n$ as $\epsilon$ goes to $0$. We now show that there are actually $r-2$ vertices of degree close to $n$.

\begin{lemma}\label{more vertices of large degree}
Assume that $1\leq k < r-2$ and that $\{v_1,\cdots, v_k\}$ are a set of vertices each with degree at least $(1-\eta)n$ and with eigenvector entry at least $1-\eta$. Then there is an absolute constant $C_4$ and a vertex $v_{k+1} \not\in \{v_1\cdots v_k\}$ such that the degree of $v_{k+1}$ is at least $(1-C_4(\eta+\epsilon))n$ and the eigenvector entry for $x_{k+1}$ is at least $1-C_4(\eta+\epsilon)$.
\end{lemma}

\begin{proof}
Let $K = \{v_1,\cdots v_k\}$. Then the eigenvector-eigenvalue equation for $A^2$ gives
\begin{align*}
&(r-2)(n-r+2) \leq \lambda^2 = \lambda^2 \mathbf{x}_z  = \sum_{v\sim z}\sum_{w\sim v} \mathbf{x}_w \leq \sum_{vw \in E(G)} (\mathbf{x}_v + \mathbf{x}_w) \\
& = \sum_{vw\in E(S)} (\mathbf{x}_v + \mathbf{x}_w) + \sum_{vw\in E(S,L)} (\mathbf{x}_v + \mathbf{x}_w)+ \sum_{vw \in E(L)} (\mathbf{x}_v + \mathbf{x}_w)\\
& \leq 2\epsilon O(n) + \sum_{vw\in E(S,L)} (\mathbf{x}_v + \mathbf{x}_w) + O(\sqrt{n}).
\end{align*}
This implies that 
\[
\sum_{\substack{vw\in E(S,L) \\ v,w\not\in K}} (\mathbf{x}_v + \mathbf{x}_w) \geq (r-2)(n-r+2) - 2\epsilon O(n) - O(\sqrt{n}) - k(1+\epsilon)n. 
\]
The definition of $K$ and Lemma \ref{K_r bipartite bound} give that the number of edges with one endpoint in $S$ and one endpoint in $L$ which is not in $K$ is at most $(r-2 + o(1))n - k(1-\eta)n$. Noting that each term in the sum is at most $\mathbf{x}_v + \epsilon$ and averaging implies that there is a vertex $v\in L\setminus K$ with the requisite eigenvector entry. Applying Lemma \ref{lemma non-neighbors in S} gives the degree condition.
\end{proof}

Starting with $z = v_1$ and iteratively applying Lemma \ref{more vertices of large degree} implies that for any $\delta > 0$, we can choose $\epsilon$ small enough that $G_r$ contains a set of $r-2$ vertices with common neighborhood of size at least $(1-\delta)n$ and with each eigenvector entry at least $1-\delta$. Since adding edges to a graph strictly increases its spectral radius, Theorem \ref{K_r clique} and $G_r$ being extremal implies that these $r-2$ vertices must form a clique. From now on, we will refer to this clique of size $r-2$ as $K$.

To complete the proof, we must show that the vertices in $K$ have degree $n-1$. Once this is proved, it implies that $G_r$ is a subgraph of the join of $K_{r-2}$ and an independent set of size $n-r+2$. But since adding any edge to this graph creates a $K_r$ minor, we will have that $G_r$ is exactly $K_{r-2}$ join an independent set of size $n-r+2$. Let $T$ be the common neighborhood of $K$ and let $R = V(G_r) \setminus (T\cup K)$.

\begin{lemma}\label{upper bound small entries}
Let $c$ be a constant such that any graph of average degree $c$ has a $K_r$ minor. Then $\epsilon$ can be chosen small enough so that if $v\in V(G_r) \setminus K$, then $\mathbf{x}_v < \frac{1}{2c}$.
\end{lemma}

\begin{proof}
Note that any vertex in $R$ can be adjacent to at most one vertex in $T$, otherwise there is a $K_r$ minor. By definition of $R$, each vertex in $R$ can be adjacent to at most $r-3$ vertices in $K$. First we give a bound on the sum of the eigenvector entries in $R$ and we use this to give a bound on each eigenvector entry.
\[
\lambda\sum_{u\in R} \mathbf{x}_u  = \sum_{u\in R}\sum_{v\sim u} \mathbf{x}_v \leq 2e(R) + (r-2)|R| = \delta O(n).
\]
That is
\[
\sum_{u\in R} \mathbf{x}_u = \delta O(\sqrt{n}).
\]
Now let $v\in V(G_r)\setminus K$. Again, note that $v$ can have at most $r-2$ neighbors in $K\cup T$. Therefore
\[
\lambda \mathbf{x}_v = \sum_{w\sim v} \mathbf{x}_w \leq r-2 + \sum_{w\in R} \mathbf{x}_w = r-2 + \delta O(\sqrt{n}).
\]
Dividing by $\lambda$ and choosing $\epsilon$ small enough to make $\delta$ small enough gives the result.
\end{proof}

Finally, we complete the proof of Theorem \ref{K_r}:

\begin{lemma}\label{B empty}
$R$ is empty.
\end{lemma}

\begin{proof}
By way of contradiction, assume that $R$ is not empty. Then there is a vertex $v$ in $R$ with at most $c$ neighbors in $R$. $v$ can be adjacent to at most $1$ vertex in $T$ and at most $r-3$ vertices in $K$. Let $u$ be a vertex in $K$ which is not adjacent to $v$. Now let $H$ be the graph obtained from $G_r$ by removing all edges incident with $v$ and then connecting $v$ to each vertex in $K$. Since $K$ induces a clique, the graph $H$ has no $K_r$ minor. Let $B$ be the adjacency matrix of $H$ and let $\mu$ be its spectral radius. Now
\[
\mu - \lambda \geq \frac{\mathbf{x}^TB\mathbf{x}}{\mathbf{x}^T\mathbf{x}} - \frac{\mathbf{x}^TA\mathbf{x}}{\mathbf{x}^T\mathbf{x}} \geq \frac{2\mathbf{x}_v}{\mathbf{x}^T\mathbf{x}} \left( \mathbf{x}_u - \sum_{\substack{vw\in E(G_r) \\ w\not\in K} }\mathbf{x}_w\right) \geq \frac{2\mathbf{x}_v}{\mathbf{x}^T\mathbf{x}} \left(\mathbf{x}_u - \frac{c+1}{2c}\right),
\]
where the last inequality follows by Lemma \ref{upper bound small entries}. Choosing $\epsilon$ small enough so that $1-\delta > \frac{c+1}{c}$ gives $H$ is a $K_r$ minor-free graph with larger spectral radius than $G_r$, a contradiction.
\end{proof}

\section{Graphs with no $K_{s,t}$ minor}\label{section K_s,t}
Let $2\leq s\leq t$ and $G_{s,t}$ be a graph on $n$ vertices with no $K_{s,t}$ minor such that the spectral radius of its adjacency matrix is at least as large as the spectral radius of the adjacency matrix of any other $n$-vertex graph with no $K_{s,t}$ minor. Throughout this section, $A$ will denote the adjacency matrix of $G_{s,t}$ and $\lambda$ its spectral radius. $\mathbf{x}$ will be the eigenvector for $\lambda$ normalized to have infinity norm equal to $1$ and $z$ will be a vertex such that $\mathbf{x}_z = 1$. First we will show that for $n$ large enough, $G_{s,t}$ is a subgraph of the join of $K_{s-1}$ and an independent set of size $n-s+1$. We omit the proof of the following proposition as it is similar to the proofs of Lemmas \ref{lower bound on lambda}--\ref{upper bound small entries}.

\begin{proposition}\label{K_s,t prop}
For any $\delta>0$, if $n$ is large enough, then $G_{s,t}$ contains a set $K$ of $s-1$ vertices which have common neighborhood of size at least $(1-\delta)n$ and each of which has eigenvector entry at least $1-\delta$. Further, for any vertex $u\in V(G_{s,t})\setminus K$, we have 
\[
\mathbf{x}_u < \frac{(1-\delta)}{c(s-1)(s-2)}
\]
where $c$ is chosen so that any graph of average degree $c$ has a $K_{s,t}$ minor.
\end{proposition}

Let $T$ be the common neighborhood of $K$ and $R = V(G_{s,t}) \setminus (T\cup K)$. First we show that $K$ induces a clique and then we show that $R$ is empty. 

\begin{lemma}\label{K clique}
$K$ induces a clique.
\end{lemma}
\begin{proof}
Assume that there are vertices $u,v\in K$ such that $u\not\sim v$. Now, Theorem \ref{K_s,t clique} guarantees that there is a set of at most $C:= \frac{c(s-1)(s-2)}{2(1-\delta)}$ edges such that we can delete these edges, make $K$ a clique, and the resulting graph will have no $K_{s,t}$ minor. Call this set of at most $C$ edges $E_1$ and call the resulting graph $H$. Let $B$ be the adjacency matrix of $H$ and $\mu$ the spectral radius of $B$. Note that all edges in $E_1$ have at least one endpoint with eigenvector entry less than $\frac{1}{2C}$ by Proposition \ref{K_s,t prop}. Then
\[
\mu - \lambda \geq \frac{2}{\mathbf{x}^T\mathbf{x}}\left( \mathbf{x}_u\mathbf{x}_v - \sum_{wy\in E_1} \mathbf{x}_w\mathbf{x}_y \right) \geq \frac{2}{\mathbf{x}^T\mathbf{x}}\left((1-\delta)^2 - C\cdot \frac{1}{2C}\right).
\]
Choosing $\delta$ small enough that $(1-\delta)^2 > 1/2$ yields $\mu > \lambda$, a contradiction. So $K$ must induce a clique.

\end{proof}

\begin{lemma}
$R$ is empty.
\end{lemma}
\begin{proof}
The proof is similar to the proof of Lemma \ref{B empty}, noting that adding a vertex adjacent to a clique of size $s-1$ to a graph with no $K_{s,t}$ minor cannot create a $K_{s,t}$ minor.
\end{proof}

So we have that the vertices of $K$ have degree $n-1$ in $G_{s,t}$. We now consider the graph induced by $V(G_{s,t}) \setminus K$. Note that if any vertex in this induced graph has $t$ neighbors, this creates a $K_{s,t}$ in $G_{s,t}$. Therefore the graph induced by $V(G_{s,t}\setminus K)$ has maximum degree at most $t-1$. 

We need an interlacing result. We comment that many times interlacing theorems are used to give a lower bound on the spectral radius of a graph via eigenvalues of either a subgraph of the graph or a quotient matrix formed from the graph. This theorem gives an {\em upper} bound on the spectral radius of a graph based on the eigenvalues of a quotient-like matrix.

\begin{theorem}\label{interlacing}
Let $H_1$ be a $d$-regular graph on $n_1$ vertices and $H_2$ be a graph with maximum degree $k$ on $n_2$ vertices. Let $H$ be the join of $H_1$ and $H_2$. Define
\[
B: = \begin{bmatrix}
d & n_2 \\
n_1 & k
\end{bmatrix}
\]
Then $\lambda_1(H) \leq \lambda_1(B)$ with equality if and only if $H_2$ is $k$-regular.
\end{theorem}
\begin{proof}
Let $A(H)$ be the adjacency matrix of $H$. Let 
\[
A(H) \begin{bmatrix}
x_1 \\ \vdots \\ x_{n_1} \\ y_1 \\ \vdots \\ y_{n_2}
\end{bmatrix} = \lambda_1(H)  \begin{bmatrix}
x_1 \\ \vdots \\ x_{n_1} \\ y_1 \\ \vdots \\ y_{n_2}
\end{bmatrix}
\]
where the eigenvector entries labeled by $x$'s correspond to vertices in $H_1$ and those by $y$'s to vertices in $H_2$. Assume that $\begin{smallmatrix}[ x_1 & \hdots & x_{n_1} & y_1 & \hdots & y_{n_2} \end{smallmatrix}]^T$ is normalized to have $2$-norm equal to $1$. Then 
\[
\lambda_1(H) = 2\sum_{ij\in E(H_1)} x_ix_j + 2\sum_{i=1}^{n_1}\sum_{j=1}^{n_2} x_iy_j + 2\sum_{ij\in E(H_2)} y_iy_j.
\]
Now note that 
\[
2\sum_{ij \in E(H_1)} x_ix_j \leq \lambda_1(H_1) \left(\sum_{i=1}^{n_1} x_i^2\right) = d \left(\sum_{i=1}^{n_1} x_i^2\right)
\]
and 
\[
2\sum_{ij\in E(H_2)} y_iy_j \leq \lambda_1(H_2) \left(\sum_{j=1}^{n_2} y_j^2\right) \leq k\left(\sum_{j=1}^{n_2} y_j^2\right).
\]
Two applications of Cauchy-Schwarz and one application of the AM-GM inequality give
\[
2\sum_{i=1}^{n_1}\sum_{j=1}^{n_2} x_iy_j \leq (n_1+n_2)\sqrt{\left(\sum_{i=1}^{n_1} x_i^2\right)}\sqrt{\left(\sum_{j=1}^{n_2} y_j^2\right)}
\]
Let $x^2 = \sum_{i=1}^{n_1} x_i^2$ and $y^2 = \sum_{j=1}^{n_2} y_j^2$. So
\[
\lambda_1(H) \leq dx^2 + (n_1+n_2)xy + ky^2.
\]
On the other hand
\[
\lambda_1(B) \geq \begin{bmatrix} x & y\end{bmatrix} B \begin{bmatrix} x \\ y\end{bmatrix} = dx^2 + (n_1+n_2)xy + ky^2.
\]
Note that if $H_2$ is not $k$-regular, then $\lambda_1(H_2) < k$, and equality cannot occur. On the other hand, if $H_2$ is $k$-regular, then the partition of $V(H)$ into $V(H_1)$ and $V(H_2)$ forms an equitable partition with quotient matrix $B$, implying that $B$ and $A(H)$ have the same spectral radius (cf \cite{Godsil}).
\end{proof}

Now we can finish the proof of Theorem \ref{K_s,t}.

\begin{proof}[Proof of Theorem \ref{K_s,t}]
We now know that $G_{s,t}$ contains as a subgraph the join of a clique of size $s-1$ (namely $K$) and an independent set of size $n-s+1$ and that the graph induced by $V(G_{s,t})\setminus K$ has maximum degree at most $t-1$. Theorem \ref{interlacing} then yields 
\[
\lambda \leq \frac{s+t-3 + \sqrt{(s+t-3)^2 + 4((s-1)(n-s+1) - (s-2)(t-1))}}{2},
\]
with equality if and only if the graph induced by $V(G_{s,t})\setminus K$ is $(t-1)$-regular. It remains to show that equality can hold if and only if $V(G_{s,t})\setminus K$ induces a disjoint union of copies of $K_t$. To accomplish this, we use a trick of Nikiforov \cite{Nikiforov}. Assume that $H$ is a connected component of the graph induced by $V(G_{s,t}\setminus K)$ on $h$ vertices. We may assume that this component is $t-1$ regular and we must show that $h = t$. If $h = t+1$, then any pair of nonadjacent vertices in $H$ have $t-1$ common neighbors. These vertices along with $K$ then form a $K_{s,t}$. 

Now assume that $h\geq t+2$. Since $H$ is dominated by $K$ and $G_{s,t}$ has no $K_{s,t}$ minor, $H$ does not have a $K_{1,t}$ minor. By \cite{ChudnovskyReedSeymour}, since $H$ is connected we have that $|E(H)| \leq h+\frac{1}{2}t(t-3)$, contradicting that $H$ is $(t-1)$-regular. Therefore $H$ must be a $K_t$, and so equality occurs if and only if $V(G_{s,t})\setminus K$ induces the disjoint union of copies of $K_t$ (implying that $n\equiv s-1 \pmod{t}$), completing the proof.
\end{proof}

We note that if $t$ does not divide $n-s+1$, then our proof only implies that the extremal graph is a subgraph of $K_{s-1}$ join an independent set of size $n-s+1$, and that the subgraph induced by the set of size $n-s+1$ has maximum degree $t-1$. We conjecture a similar construction is extremal when $t$ does not divide $n-s+1$.

\begin{conjecture}
Let $2\leq s\leq t$, and let $0\leq p < t$. Let $n = s-1 + kt + p$. For $n$ large enough, the $n$-vertex graph of maximum spectral radius which does not contain $K_{s,t}$ as a minor is the join of $K_{s-1}$ and $(kK_t + K_p)$.
\end{conjecture}

\section{Graphs with $\mu(G) \leq m$}\label{section CdV}
Let $m$ be a positive integer. Let $G_m$ be a graph on $n$ vertices, with $\mu(G_m) \leq m$, which has the largest spectral radius of its adjacency matrix over all $n$-vertex graphs with Colin de Verdi\`ere parameter at most $m$. Throughout this section, $A$ will denote the adjacency matrix of $G_m$, which will have spectral radius $\lambda$. $\mathbf{x}$ will be an eigenvector for $\lambda$ with infinity norm $1$. We will use the following theorem of van der Holst, Lov\'asz, and Schrijver \cite{survey}.

\begin{theorem}\label{join theorem}
Let $G = (V,E)$ be a graph and let $v\in V$. Then
\[
\mu(G) \leq \mu(G - v) + 1.
\]
If $v$ is connected to all other nodes, and $G$ has at least one edge, then equality holds.

\end{theorem}

The main results of \cite{TaitTobin} show that for $n$ large enough, the outerplanar graph of maximum spectral radius is $K_1$ join $P_{n-1}$ and the planar graph of maximum spectral radius is $K_2$ join $P_{n-2}$. Since a graph $G$ is outerplanar if and only if $\mu(G) \leq 2$ and planar if and only if $\mu(G) \leq 3$, Theorem \ref{CdV} is proved if $m\in \{2,3\}$, and we will from now on assume $m\geq 4$. Since $K_2$ join $P_{n-2}$ is planar, Theorem \ref{join theorem} implies that the join of $K_{m-1}$ and $P_{n-m+1}$ has Colin de Verdi\`ere parameter equal to $m$ for any $m$. 

We also note that since $\mu(K_{m, m}) = m+1$ (cf \cite{survey}), our graph $G_m$ cannot contain $K_{m,m}$ as a minor. We omit the proof of the following Proposition, as it is similar to the proofs Lemmas \ref{lower bound on lambda}--\ref{upper bound small entries}.

\begin{proposition}
For any $\delta > 0$, if $n$ is large enough, then $G_m$ contains a set $K$ of $m-1$ vertices which have a common neighborhood of size at least $(1-\delta)n$ and each of which has eigenvector entry at least $1-\delta$. Further, for any vertex $u\in V(G_m)\setminus K$, we have 
\[
\mathbf{x}_u < \frac{1-\delta}{c(s-1)(s-2)}
\]
where $c$ is chosen so that any graph of average degree $c$ has Colin de Verdi\`ere parameter at least $m+1$.
\end{proposition}

Let $T$ be the common neighborhood of $K$ and $R = V(G_m) \setminus (T\cup K)$. We show next that $K$ induces a clique and that $R$ is empty.

\begin{lemma} $K$ induces a clique.
\end{lemma}
\begin{proof} The proof is similar to the proof of Lemma \ref{K clique} once we note that for any integer $m$ the property that $\mu(G) \leq m$ can be characterized by a finite family of excluded minors \cite{survey}.
\end{proof}

\begin{lemma}
$R$ is empty.
\end{lemma}
\begin{proof}
The proof is similar to the proof of Lemma \ref{B empty}, we must only check that if $H$ is a graph with $\mu(H) = m$, then adding a new vertex adjacent to a clique of size $m-1$ does not increase the Colin de Verdi\`ere parameter. But this follows since adding a new vertex adjacent to a clique of size $m-1$ is an $(m-1)$-clique sum (cf \cite{survey}).
\end{proof}

We can now complete the proof of Theorem \ref{CdV}.

\begin{proof}[Proof of Theorem \ref{CdV}]
We now know that $G_m$ contains as a subgraph of the join of $K = K_{m-1}$ and an independent set of size $n-m+1$. Let $H$ be the graph induced by $V(G_m) \setminus K$. 

First we claim that $H$ has maximum degree $2$. In order to see this, we note that $\mu(K_{1,3}) = 2$. This and Theorem \ref{join theorem} imply that the join of $K_{1,3}$ and $K_{m-1}$ has Colin de Verdi\`ere parameter $m+1$. Therefore, this graph cannot be a subgraph of $G_m$, and so $H$ can not have a vertex of degree $3$ or more.

Therefore, $H$ is the disjoint union of paths and cycles. But we now claim that $H$ cannot contain any cycles, as a cycle is a $K_3$ minor. A $K_3$ minor joined to a $K_{m-1}$ is a $K_{m+2}$ minor, which violates $\mu(G_m) \leq m$. 

So now $H$ induces a disjoint union of paths, which means that $G_m$ is a subgraph of the join of $K_{m-1}$ and $P_{n-m+1}$. By the Perron-Frobenius Theorem and maximality of $\lambda$, $G_m$ must be exactly equal to the join of $K_{m-1}$ and $P_{n-m+1}$.
\end{proof}

\section*{Acknowledgements}
We would like to thank Maria Chudnovsky and Paul Seymour; most of the results in section \ref{structural section} were found in collaboration with them.


\begin{thebibliography}{00}
\bibitem{BabaiGuiduli} L.\ Babai and B.\ Guiduli, Spectral extrema for graphs: the Zarankiewicz problem, {\em Electronic J. Combin.} {\bf 16} (2009) R123.

\bibitem{barioli survey} F.\ Barioli, W.\ Barrett, S.\ Fallat, H.\ T.\ Hall, L.\ Hogben, B.\ Shader, P.\ van den Driessche, and H.\ van der Holst, Parameters related to tree-width, zero forcing, and maximum nullity of a graph, {\em Journal of Graph Theory} {\bf 72} (2013) 146--177.

\bibitem{ChudnovskyReedSeymour} M.\ Chudnovsky, B.\ Reed, and P.\ Seymour, The edge-density for $K_{2,t}$ minors, {\em J.\ Combin.\ Theory Ser.\ B} {\bf 101} (2011), 18--46.

\bibitem{Colin1} Y.\ Colin de Verdi\`ere, Sur un nouvel invariant des graphes et un crit\`ere de planarit\'e, {\em J.\ Combin.\ Theory Ser.\ B} {\bf 50} (1990) 11--21.

\bibitem{Godsil} C.\ D.\ Godsil, {\em Algebraic Combinatorics}, Chapman and Hall, New York (1993).

\bibitem{Hoffman} A.\ J.\ Hoffman, On eigenvalues and colorings of graphs, {\em Graph Theory and Its Applications} (Proc. Advanced Sem., Math Research Center, Univ. of Wisconsin, Madision, Wis., 1969), New York.

\bibitem{survey} H.\ van der Holst, L.\ Lov\'asz, and A.\ Schrijver, The Colin de Verdi\`ere graph parameter, {\em Bolyai Soc. Math. Stud.} {\bf 7}, J\'anos Bolyai Math. Soc. Budapest (1999) 29--85.

\bibitem{Hong} Y.\ Hong, Tree-width, clique-minors, and eigenvalues, {\em Discrete Math.} {\bf 274} (2004), 281--287.

\bibitem{bipartite rigidity} G.\ Kalai, E.\ Nevo, and I.\ Novik, Bipartite rigidity, {\em Trans. Amer. Math. Soc.} {\bf 368} (2016) 5515--5545.

\bibitem{Kostochka1} A.\ V.\ Kostochka, The minimum Hadwiger number for graphs with a given mean degree of vertices, {\em Metody Diskret. Anal.} {\bf 38} (1982), 37--58.

\bibitem{Kostochka2} A.\ V.\ Kostochka, Lower bound for the Hadwiger number for graphs by their average degree, {\em Combinatorica} {\bf 4} (1984), 307--316.

\bibitem{KostochkaPrince} A.\ V.\ Kostochka and N.\ Prince, On $K_{s,t}$ minors in graphs of given average degree, {\em Discrete Math.} {\bf 308} (2008), 4435--4445.

\bibitem{LS} L.\ Lov\'asz and A.\ Schrijver, A Borsuk theorem for antipodal links and a spectral characterization of linklessly embeddable graphs, {\em Proc. Amer. Math. Soc.} {\bf 126} (1998) 1275--1285.

\bibitem{Mader} W.\ Mader, Homomorphieeigenschaften und mittlere Kantendichte von Graphen, {\em Math. Ann.} {\bf 174} (1967), 265--268.

\bibitem{Mader2} W.\ Mader, Homomorphies\"atze f\"ur Graphen, {\em Math. Ann.} {\bf 178} (1968), 154--168.

\bibitem{Myers} J.\ S.\ Myers, The extremal function for unbalanced bipartite minors, {\em Discrete Math.} {\bf 271} (2003), 209--222.

\bibitem{Nevo} E.\ Nevo, Embeddability and stresses of graphs, {\em Combinatorica} {\bf 27} (2007) 465--472.

\bibitem{NikiforovKST} V.\ Nikiforov, A contribution to the Zarankiewicz problem, {\em Linear Algebra Appl.} {\bf 432} (2010) 1405--1411.

\bibitem{NikiforovTuran} V.\ Nikiforov, Some inequalities for the largest eigenvalue of a graph, {\em Combinatorics, Probability and Computing} {\bf 11} (2002) 179--189.

\bibitem{Nikiforov} V.\ Nikiforov, The spectral radius of graphs with no $K_{2,t}$ minor, preprint (arXiv:1703.01839).

\bibitem{AlonBoppana} A.\ Nilli, On the second eigenvalue of a graph, {\em Discrete Math.} {\bf 91} (1991) 207--210.

\bibitem{RST} N.\ Robertson, P.\ Seymour, and R.\ Thomas, A survey of linkless embeddings, {\em Graph Structure Theory} (N.\ Robertson, P.\ Seymours, eds.), Contemporary Mathematics, American Mathematical Society, Providence, Rhode Island (1993) 125--136.

\bibitem{Stanley} R.\ P.\ Stanley, A bound on the spectral radius of graphs with $e$ edges, {\em Linear Algebra Appl.} {\bf 87} (1987), 267--269.

\bibitem{TaitTobin} M.\ Tait and J.\ Tobin, Three conjectures in extremal spectral graph theory, {\em J.\ Combin.\ Theory Ser.\ B}, {\bf 126} (2017), 137--161.

\bibitem{Thomason3} A.\ Thomason, Disjoint complete minors and bipartite minors, {\em European Journal of Combinatorics} {\bf 28} (2007), 1779--1783.

\bibitem{Thomason1} A.\ Thomason, The extremal function for contractions of graphs, {\em Math. Proc. Cambridge Phil. Soc.} {\bf 95} (1984), 261--265.

\bibitem{Thomason2} A.\ Thomason, The extremal function for complete minors, {\em J.\ Combin.\ Theory Ser. B} {\bf 81} (2001), 318--338.
\end{thebibliography}
\end{document}